\documentclass[12pt,a4paper,reqno,oneside]{amsart}
\usepackage{t1enc}
\usepackage{times}
\usepackage{amssymb}
\usepackage[mathscr]{euscript}
\usepackage{graphicx}
\usepackage{hyperref}
\usepackage{color}
\usepackage{float}
\usepackage{epstopdf}

\newtheorem{thm}{Theorem}[section]
\newtheorem{prop}[thm]{Proposition}
\newtheorem{lem}[thm]{Lemma}
\newtheorem{cor}[thm]{Corollary}

\theoremstyle{remark}

\theoremstyle{definition}
\newtheorem{defn}[thm]{Definition}

\newcommand*{\rom}[1]{\expandafter\@slowromancap\romannumeral #1@} 
\renewcommand{\phi}{\varphi} 
\newcommand{\E}{\mathrm{E}} 
\newcommand{\A}{\mathfrak A} 
\newcommand{\F}{\mathcal F} 
\newcommand{\R}{\mathbb R} 
\newcommand{\sgn}{\mathrm{sgn}} 

\newcommand{\filt}{(\F_t)_{t\geq 0}} 
\newcommand{\fraum}{(\Omega,\A,\filt,P)} 
\newcommand{\nooutput}[1]{}
\newcommand{\sign}{\mathrm{sign}} 

\begin{document}

\title[H\"{o}lder continuous densities of solutions of SDEs with singular drift]{H\"{o}lder continuous densities of solutions of SDEs with measurable and path dependent drift coefficients.}
\date{\today}

\author[Banos]{David Ba\~nos}
\address[David Ba\~nos]{\\
Institute of Mathematics of the University of Barcelona \\
University of Barcelona\\
Gran Via de les Corts Catalanes 585\\
08007 Barcelona, Spain}
\email[]{davidru@math.uio.no}
\author[Kr\"uhner]{Paul Kr\"uhner}
\address[Paul Kr\"uhner]{\\
FAM - financial and actuarial mathematics \\
Technical University Vienna\\
Wiedner Hauptstrasse 8-10\\
1040 Vienna, Austria}
\email[]{paul.kruehner@fam.tuwien.ac.at}



\begin{abstract}
We consider a process given as the solution of a one-dimensional stochastic differential equation with irregular, path dependent and time-inhomogeneous drift coefficient and additive noise. H\"{o}lder continuity of the Lebesgue density of that process at any given time is achieved using a different approach than the classical ones in the literature. Namely, the H\"{o}lder regularity of the densities is obtained via a control problem by identifying the stochastic differential equation with the \emph{worst} global H\"{o}lder constant. Then we generalise our findings to a larger class of diffusion coefficients. The novelty of this method is that it is not based on a variational calculus and it is suitable for non-Markovian processes. 
\end{abstract}

\maketitle

\section{Introduction}

The examination of densities of random variables has been an active area of research during the last decades for its applications and in its own interest. To find criteria for the Lebesgue density of an absolutely continuous random variable to be regular has been the aim over the past years. P.\ Malliavin in \cite{malliavin.78} was interested in providing a probabilistic proof of L.\ H\"{o}rmander's theorem, see L.\ H\"{o}rmander \cite{hormander.69}, which, in short, is about a sufficient condition for an operator to be \emph{hypoelliptic}. He believed that H\"{o}rmander's condition implies that the finite dimensional Lebesgue densities of a solution of a stochastic differential equation (SDE) are smooth. For this reason, he developed a stochastic calculus of variations in order to provide the concept of derivative of a random variable in a certain sense. For instance, a classical result on this matter is that if the coefficients of an SDE are smooth with bounded derivatives and the so-called H\"{o}rmander's condition mentioned above holds, then the solution is smooth in the Malliavin sense at any time. Then, it is shown in \cite{malliavin.78} that Malliavin smoothness together with a non-degeneracy condition implies that the laws of the solutions are absolutely continuous with respect to the Lebesgue measure and the densities are smooth. Classical results based on analytic methods can be found in \cite{friedman.83} and \cite{evans.10}.

A different approach is credited to N.\ Bouleau and F.\ Hirsch \cite{bouleau.hirsch.86} where they show absolute continuity of the finite dimensional laws of solutions to SDEs based on a stochastic calculus of variations in finite dimensions utilising a limit argument. Also, as a motivation of \cite{bouleau.hirsch.86}, D.\ Nualart and M.\ Zakai \cite{nualart.zakai.89} found related results on the existence and smoothness of conditional densities of Malliavin differentiable random variables.

Further research has been carried out, let us mention some achievements on this topic. S.\ De Marco \cite{de.marco.11} shows local smoothness of densities on an open domain under the usual condition of ellipticity and that the coefficients are smooth on such domain. Also, V.\ Bally and A.\ Kohatsu-Higa \cite{bally.kohatsu-higa.10} show that the densities of a type of a two-dimensional degenerated SDE are bounded and they provide both upper and lower bounds, for this case, it is assumed that the coefficients are five times differentiable with bounded derivatives. Let us also mention the advances of V.\ Bally and L.\ Caramellino \cite{bally.caramellino.11} where an integration by parts formula (IPF) is derived and the integrability of the weight obtained in the formula gives the desired regularity of the density. As a consequence of the aforementioned result D.\ Ba\~{n}os and T.\ Nilssen \cite{banos.nilssen.14} give a condition to obtain regularity of densities of solutions to SDEs according to how regular the drift is. The technique is also based on Malliavin calculus and a sharper estimate on the moments of the derivative of the flow associated to the solution. This result is a slight improvement of a very similar condition obtained by S.\ Kusuoka and D.\ Stroock in \cite{kusuoka.stroock.82} when the diffusion coefficient is constant and the drift may be unbounded. Last but not least, we also cite the results by A.\ Kohatsu-Higa and A.\ Makhlouf \cite{kohatsu.makhlouf.13} where the authors show smoothness of the density for smooth coefficients that may also depend on an external process whose drift coefficient is irregular.

It appears to be impossible to deduce (optimal) regularity properties of densities of solutions of SDEs with irregular coefficients, e.g.\ non-Lipschitz drift coefficient. Nevertheless, some results in this direction have been obtained. For example, M.\ Hayashi, A.\ Kohatsu-Higa and G.\ Y\^{u}ki in \cite{hayashi.et.al.12} prove that SDEs with bounded H\"{o}lder continuous drift and smooth elliptic diffusion coefficients admit H\"{o}lder continuous densities at any time. Their method is also based on Malliavin calculus in connection with IPFs and estimates on the characteristic function of the solution.

We remark that all preceding works are based upon Malliavin calculus in connection with IPFs or Fourier analysis.

On the contrary, as a recent new technique, originally developed in \cite{Fournier.Printems.10}, there are related results where Malliavin calculus is not directly employed. A.\ Debussche and N.\ Fournier \cite{debussche.fournier.13} prove that the finite dimensional densities of a solution of an SDE with jumps lies in a certain (low regular) Besov space when the drift is H\"{o}lder continuous. The very related result by Hayashi, Kohatsu-Higa and Y\^{u}ki \cite{hayashi.et.al.14} shows that the density of the marginals of an SDE are H\"older continuous of some degree which depends on the Fourier regularity of the drift coefficient. The result requires that the SDE has a Markovian structure and the drift is bounded with its Fourier transform belonging to some Sobolev space. Their methodology is essentially based on finding estimates for the Fourier-Stiltjes transform of the finite dimensional laws of the solution. This paper improves the results in \cite{hayashi.et.al.14} in the one-dimensional case in two ways. First, we only need boundedness of the drift coefficient while Hayashi et.\ al \cite{hayashi.et.al.14} need additional Fourier regularity. Second, we do allow for path-dependence in the drift or even dependence on any other adapted process.

It is therefore important to emphasise that in this paper we do \emph{not} use Malliavin calculus or any other type of variational calculus. We show that It\^{o} processes with additive noise and merely bounded and measurable drift admit H\"{o}lder continuous densities of any order strictly less than one. In other words, we look at
\begin{align}\label{e: intro}
X_u(t) := x+\int_0^t u(s) ds + W(t), \quad t\in [0,T], \quad x\in \R,
\end{align}
where $W$ is a standard Brownian motion and $u$ is a progressively measurable stochastic process bounded by some constant $C>0$. Note that this includes SDEs where the drift may depend on the solution in a non-Markovian manner.

Our main technique is based on a worst-case study by employing optimal stochastic control. First we restrict to the class of controlled processes like \eqref{e: intro} whose drift coefficient is bounded by some constant $C>0$ and we look at a time $t>0$, no further regularity is assumed. Then, for given $t>0$ one seeks to find a \emph{worst-case} process which maximises the global H\"{o}lder constant of the density of $X_u(t)$ at time $t$ among all members in our class. Hence, one is reducing the overall problem to studying a specific case, namely the process in \eqref{e: intro} associated to the optimal control. If one is able to show that this \emph{optimal} process admits marginal H\"{o}lder continuous densities at any time $t>0$, then this implies that any other process $X_u(t)$, $t>0$, has H\"{o}lder continuous densities of any order $\alpha\in (0,1)$, as well. Nevertheless, solving the proposed stochastic control problem is not an easy task. In order to circumvent this difficulty we proceed in a slightly different way. We simply pick a specific control $\widetilde{u}$ in the class of allowed controls and compute \emph{how well} it performs compared to any other control, including the optimal one.

This idea is inspired by a previous work, see \cite{banos.kruehner.15}, to find the optimal lower and upper bounds for densities of It\^{o} type processes with additive noise. Our method is robust since no well-behaviour on the drift is needed other than merely boundedness and no Markovianity of the system is assumed, hence overcoming limitations of Malliavin calculus.

Moreover, it is known that the density of $X(t)$, $t>0$, given by
$$ X(t) = -\int_0^t \sgn(X(s)) ds + W(t),\quad t\geq 0,$$
where $W$ is a standard Brownian motion is globally Lipschitz continuous but not more, meaning that the regularity obtained is almost optimal, see e.g. \cite[Theorem 3.5]{banos.kruehner.15}.

This paper is organised as follows. In Section \ref{main results} we summarise our main results with some generalisations to non-trivial diffusion coefficients. In Section \ref{control problem} we pose the stochastic control problem of interest, then we provide an \emph{error} estimate for the performance of a selected control compared to any other admissible control. Finally, we prove the H\"{o}lder continuity of the density of $X_u(t)$ at any given time $t>0$. To conclude the paper, we summarise in two appendices the technical results needed in Section \ref{control problem}.

\subsection{Notations}
For a differentiable function $f:(0,T)\times \R \rightarrow \R$, $(t,x)\mapsto f(t,x)$, we denote by $\partial_1 f$, resp. $\partial_2 f$, the derivatives with respect to its first argument, resp. its second argument. For an open subset $U\subset \R$ and a function $f:(0,T)\times U\rightarrow \R$, $(t,x)\mapsto f(t,x)$, we say that $f$ is of class $C^{1,2}$ if $f\in C^{1,2}((0,T)\times U)$ being $f\in C^{1,2}((0,T)\times U)$ the space of continuous functions on $(0,T)\times U$ such that the partial derivatives $\partial_1 f$, $\partial_2 f$ and $\partial_2^2 f$ exist and are continuous. The notation $C^{1,2}([0,T)\times U)$ means that $f\in C^{1,2}((0,T)\times U)$ and the indicated partial derivatives have continuous extensions to $[0,T)\times U$. 
We denote the \emph{signum} function by $\sgn(x):=1_{\{x > 0\}}-1_{\{x\leq 0\}}$ for any $x\in\mathbb R$. Finally, we denote by $\phi$, respectively, $\Phi$, the density, respectively distribution function, of a standard normal random variable.

Further notations are used as in \cite{js.87}.

\section{Main results}\label{main results}
In this section we present our main result and some of its consequences. In particular, we will show that densities of the finite dimensional laws of the solution of an SDE with additive noise in the one-dimensional case are H\"{o}lder continuous of order $\alpha\in (0,1)$ and give some extensions to more general diffusion coefficients.

Throughout this section let $\fraum$ be a filtered probability space with the usual assumptions on the filtration $\mathcal{F}=(\mathcal{F}_t)_{t\geq 0}$, i.e.\ $\mathcal F_0$ contains all $P$-null sets and $\mathcal F$ is right-continuous, $W$ be a one-dimensional standard $\mathcal F$-Brownian motion.

The next results constitute the core result of this paper and will be proven in detail in the next section.

\begin{thm}\label{t:main theorem}
  Let $\beta$ be a bounded progressively measurable process with values in $\mathbb R$, $\alpha\in(0,1)$ and $X$ a stochastic process such that $dX(t) = \beta(t) dt + dW(t)$.
  
  Then $X(T)$ has $\alpha$-H\"older continuous density for any $T>0$. Furthermore, if $\beta$ is bounded by $K>0$, then the continuity constant is at most
  {\footnotesize $$ C^K_\alpha(T) := K^{1+\alpha}\left(\frac{1}{\sqrt{2\pi e^\alpha}(TK^2)^{(1+\alpha)/2}} + \frac{4}{\sqrt{2\pi e^\alpha}} \int_0^{TK^2} \left(\frac{\phi(\sqrt s)}{\sqrt{s}}+\Phi(\sqrt{s})\right) \frac{1}{(TK^2-s)^{(1+\alpha)/2}}ds\right)$$}
\end{thm}
\begin{proof}
 Let $K$ be the uniform bound for $\beta$. W.l.o.g.\ $K>0$. Define
  \begin{align*}
   B(t) &:= KW(t/K^2), \\
   u(t) &:= \beta(t/K^2)/K, \\
   Y(t) &:= KX(t/K^2) \\
        &= KX(0) + \int_0^t u(s) ds + B(t).
  \end{align*}
  Then $Y(TK^2)$ admits the assertions of Proposition \ref{p: control problem} below and hence $Y(TK^2)$ has $\alpha$-H\"older continuous density. Clearly, $X(T)=Y(TK^2)/K$ has $\alpha$-H\"older continuous density as well with the constant given above. 
\end{proof}

We now focus on two simple extensions of our main result which utilise It\^o's formula to allow for a non-constant diffusion coefficient. In the next corollary one could allow for a path-dependent function $\sigma$ instead by using the pathwise It\^o-formula, cf.\ \cite[Theorem 4.1]{cont.fournie.13}. However, in order to not introduce heavy notation we rely on the classical It\^{o}-formula.
\begin{cor}
 Let $X$ be an $\mathcal F$-adapted, $\mathbb R$-valued stochastic process such that
  $$X(t) = x_0+ \int_0^t\beta(s)ds + \int_0^t\sigma(s,X(s))dW(s), \quad t\geq0$$
 where $x_0\in\mathbb R$, $\beta$ is progressively measurable process, $\sigma:\mathbb R_+\times \mathbb R\rightarrow (0,\infty)$ is a continuously differentiable function and assume that
  $$ u(t) := \frac{\beta(t)}{\sigma(t,X(t))}-\frac{\partial_2\sigma(t,X(t))}{2}-\int_0^{X(t)}\frac{\partial_1\sigma(t,x)}{\sigma^2(t,x)}dx,\quad t\geq 0 $$
is a bounded process.

Then $X(t)$ has Lebesgue density $\rho_t$ which is H\"{o}lder continuous of any order $\alpha\in (0,1)$.
\end{cor}
\begin{proof}
 Define $F(t,x):=\int_0^x \frac{1}{\sigma(t,z)}dz$ for any $x\in\mathbb R$, $t>0$ and $Y(t):=F(t,X(t))$ for any $t\geq 0$. Then, It\^o's formula yields
 $$ Y(t) = F(0,x_0) + \int_0^tu(s)ds + W(t).$$
Hence, $Y(t)$ has H\"older-continuous density by Theorem \ref{t:main theorem}. The claim follows because for $t>0$ we have $X(t) = (F(t,\cdot))^{-1}(Y(t))$ and $F(t,\cdot)$ is an invertible and twice continuously differentiable function.
\end{proof}

Sometimes, processes live on a half line or an interval. If the coefficients behave nicely enough, then our result still applies.
\begin{cor}
  Let $-\infty\leq c< d\leq \infty$, $x_0\in(c,d)$, $b:(c,d)\rightarrow \mathbb R$ be measurable, $\sigma:(c,d)\rightarrow (0,\infty)$ be Lipschitz continuous and assume that there is a constant $K>0$ such that $|b(x)|\leq K\sigma(x)$ for any $x\in (c,d)$.

 Then $X(t)$ has $\alpha$-H\"older continuous density for any $\alpha\in(0,1)$ where $X$ is any $\mathcal F$-adapted, $(c,d)$-valued process with $$X(t)=x_0 + \int_0^t b(X(s))ds+\int_0^t \sigma(X(s))dW(s),\quad t\geq 0.$$
\end{cor}
\begin{proof}
  Define $F(y) := \int_{x_0}^y\frac{1}{\sigma(z)}dz$ for any $y\in(c,d)$, $Y(t) := F(X(t))$ and denote a bounded version of the absolutely continuous derivative of $\sigma$ by $\sigma'$. Let $t>0$. Then, It\^o-Tanaka's formula \cite[Theorem VI.1.1]{revuz.yor.99} yields
   $$ Y(t) = F(x_0) + \int_0^t\left(\frac{b(X(s))}{\sigma(X(s))}-\frac{1}{2}\sigma'(X(s))\right)ds + W(t),\quad t\geq 0.$$
   The process $u(t) := \frac{b(X(t))}{\sigma(X(t))}-\frac{1}{2}\sigma'(X(t))$ is bounded by assumptions and hence Theorem \ref{t:main theorem} yields that $Y(t)$ has $\alpha$-H\"older continuous density $\rho_{Y(t)}$ for any $\alpha\in(0,1)$. We have
    $$ \rho_{X(t)}(x) = \rho_{Y(t)}(F(x)) F'(x) $$
  where $\rho_{X(t)}$ denotes the density of $X(t)$ and hence $X(t)$ has $\alpha$-H\"older continuous density for any $\alpha\in(0,1)$.
\end{proof}

\section{A control problem}\label{control problem}
Throughout this section, let $\fraum$ be a filtered probability space with the usual assumptions on the filtration $\mathcal{F} = (\mathcal{F}_t)_{t\geq 0}$, i.e. $\mathcal{F}_0$ contains all $P$-null sets and $\mathcal{F}$ is right-continuous. Let $W$ be a one-dimensional $\mathcal{F}$-Brownian motion and define the process class
\begin{align*}
 \mathcal A &:= \{u: u\text{ is an }\mathbb R\text{-valued progressively measurable process bounded by }1\}
\end{align*}

We want to study processes of the form
\begin{align}\label{e: process}
X^x_u(t) := x+\int_0^t u(s)ds +W(t), \quad t \in [0,T],\quad x\in\R,
\end{align}
where $u\in \mathcal{A}$ and $T>0$ is some fixed time horizon. Since the process $u$ in \eqref{e: process} is bounded a simple application of Girsanov's theorem in connection with Novikov's condition guarantees existence of a density $\rho_{u,t,x}:\R \rightarrow \R$ of $X^x_u(t)$ at time $t\in(0,T]$ and starting point $x\in\mathbb R$ for any $u\in\mathcal A$.

We would like to find the process $u^\ast \in \mathcal{A}$ such that $X^x_{u^\ast}(T)$ has the density $\rho_{u^\ast,T,x}$ with the \emph{worst} H\"{o}lder constant among all densities $\rho_{u,T,x}$, $u\in \mathcal{A}$ which is a value in $[0,\infty]$. Then proving that $\rho_{u^\ast,T,x}$ has a H\"{o}lder continuous density yields that all $\rho_{u,T,x}$ also are H\"{o}lder continuous for any $u\in\mathcal A$.

An application of Proposition \ref{p:Holder condition} allows us to restate the question for H\"older continuous densities in form of the function
 $$ J_{h,k}(x) := 1_{\{x\in[k,k+h]\}}-1_{\{x\in[0,h]\}},\quad x\in\mathbb R,\quad 0<h\leq k. $$
 
Girsanov's theorem \cite[Theorem VIII.1.4]{revuz.yor.99} yields that $X^x_u(T)$ has absolutely continuous distribution function $F$. The next lemma connects H\"older continuity of $F'$ to properties of the random variable $X^x_u(T)$.
\begin{lem}
  Let $u\in\mathcal A$, $x\in\mathbb R$ and $\alpha\in(0,1]$. Then $X^x_u(T)$ has $\alpha$-H\"older continuous density if and only if there is $C>0$ such that 
   $$ \frac{\E[ J_{h,k}(X^x_u(T)) ]}{hk^\alpha} \leq C $$
   for any $0<h\leq k$.
\end{lem}
\begin{proof}
  This is immediate from Proposition \ref{p:Holder condition}.
\end{proof}

In the following argumentation we seek to maximise the expression $\E[ J_{h,k}(X^x_u(T))]$ over $x\in\mathbb R$ and $u\in\mathcal A$ for fixed $0<h\leq k$. In other words, we try to analyse the control problem
 \begin{align}\label{e: CP}
  \sup_{x\in\mathbb R,u\in\mathcal A} \E[ J_{h,k}(X^x_u(T))] = \E[ J_{h,k}(X^{x^\ast}_{u^\ast}(T))]
\end{align}  
 for some optimal starting location $x^\ast$ and optimal control $u^\ast$ where $0<h\leq k$ are fixed parameters. We like to mention, although not used in this article, that existence of an optimal control and optimal starting location can be shown by applying general theory. For the general theory of control problems we relate to {\O}ksendal and Sulem \cite{oeksendal.sulem.07}.

However, it appears to be difficult to solve the control problem in \eqref{e: CP} explicitly. This is why, we proceed differently and we simply pick a trivial control, namely $u=-1$, and compare its performance to arbitrary controls.

Before going to the main theorem of this section we need a couple of intermediate results which will be used in the proof. The first provides explicit bounds for the density of processes like \eqref{e: process} and the proof can be found in \cite{banos.kruehner.15}. The second is a key estimate for the quantity $E[J_{h,k}(X^x_u(T))]$ in terms of a bang-bang Markovian control $\widetilde u\in \mathcal{A}$, i.e.\ with $\widetilde u(t)=v(t,X^x_{\widetilde u}(t))$ where $v:\mathbb R_+\times\mathbb R\rightarrow\{-1,1\}$ is some measurable function.

\begin{prop}\label{p: BATOG}
Let $C>0$ and $u\in\mathcal A$. Then $X(t):=\int_0^tCu(s)ds + W(t)$ has Lebesgue density and one of its versions is given by
 $$\rho_t(x) := \limsup_{\epsilon\rightarrow 0}\frac{P(\vert X(t)-x\vert \leq \epsilon)}{2\epsilon},\ x\in\mathbb R.$$
Moreover, $\rho_t$ satisfies
$$0< \alpha_{t,C}(x) \leq  \rho_t(x) \leq \beta_{t,C}(x) \leq \beta_{t,C}(0) $$
for any $t>0$, $x\in\mathbb R$ where

\begin{align*}
 \alpha_{t,C}(0) &= \frac{1}{\sqrt{t}}\varphi\left(C\sqrt{t}\right)-C\Phi\left(-C\sqrt{t}\right),\quad\text{and} \\
 \beta_{t,C}(0) &= \frac{1}{\sqrt{t}}\varphi\left(C\sqrt{t}\right)+C\Phi\left(C\sqrt{t}\right)
\end{align*}
and $\Phi$ denotes the distribution function of the standard normal law and $\phi$ its density function. For $x\in\mathbb R\backslash\{0\}$ we have
\begin{align*}
  \alpha_{t,C}(x) &=\int_0^{tC^2} C\alpha_{tC^2-s,C}(0)\rho_{\theta_0^{Cx}}(s)ds\quad\text{and}\\
  \beta_{t,C}(x) &=\int_0^{tC^2}   C\beta_{tC^2 -s,C}(0) \rho_{\tau_0^{Cx}}(s)ds
\end{align*}
where 
\begin{align*}
 \rho_{\tau_0^x}(s)   &= \frac{|x|}{\sqrt{2\pi s^3}}e^{-\frac{(|x|-s)^2}{2s}}\quad{and}\\
 \rho_{\theta_0^x}(s) &= \frac{|x|}{\sqrt{2\pi s^3}}e^{-\frac{(|x|+s)^2}{2s}}
\end{align*}
for any $s>0$.
\end{prop}
\begin{proof}
  See \cite[Theorems 2.1, 2.2]{banos.kruehner.15}.
\end{proof}

We now state an error estimate for bang-bang controls.
\begin{prop}\label{p: error estimate}
Let $0<h\leq k$ be fixed and $v:\mathbb R_+\times\mathbb R\rightarrow \{-1,1\}$ be measurable. Let $Z^{s,x}$ be the unique strong solution to the stochastic differential equation $$Z^{s,x}(t) := x + \int_s^t v(r,Z^{s,x}(r))dr + W(t)-W(s)$$ for $x\in\mathbb R$, $0\leq s\leq t\leq T$, cf. \cite[Theorem IX.3.5]{revuz.yor.99}, and assume that the function
 $$ \widetilde V(s,x) := \E[J_{h,k}(Z^{s,x}(T))] $$
is of class $C^{1,2}$ on $[0,T)\times\mathbb R$.

Then for any $x\in\mathbb R$, $u\in\mathcal A$ with 
 $$ \int_0^T \E[ (\partial_2\widetilde V(s,X^x_u(s))^2 ] ds < \infty$$
we have
\begin{align}\label{e: error}
 E[J_{h,k}(X^x_u(T))]\leq \widetilde V(0,x) + 2\int_0^T \int_{\R}\beta_{s,1}(0) 1_{\{v(s,z)\neq  \sgn(\partial_2 \widetilde V(s,z))\}}|\partial_2 \widetilde V(s,z)|dz  ds,
\end{align}
where $\beta_{s,1}(0)=\frac{\phi(\sqrt{s})}{\sqrt{s}}+\Phi(\sqrt{s})$ is the uniform bound given in Proposition \ref{p: BATOG}.
\end{prop}
\begin{proof}
Let $u\in\mathcal A$, $x\in\mathbb R$ such that 
 $$ \int_0^T \E[ (\partial_2\widetilde V(s,X^x_u(s))^2 ] ds < \infty. $$
Then, $M(t) := \int_0^t \partial_2 \widetilde V(s,X^x_u(s)) dW(s)$ for $t\in[0,T]$ is a martingale.

The function $\widetilde V$ satisfies the Kolmogorov backward equation

\begin{align}\label{e: kolmogorov}
\partial_1 \widetilde{V}(s,z) + \partial_2 \widetilde{V}(s,z) v(s,z) +\frac{1}{2}\partial_2^2 \widetilde{V}(s,z)=0
\end{align}
for any $s\in [0,T)$ and $z\in \R$.

On the other hand, observe that
$$E[J_{h,k}(X^x_u(T))] = E[\widetilde{V}(T,X^x_u(T))].$$

Since $\widetilde{V}\in C^{1,2}([0,T)\times \R)$ we can apply It\^{o}'s formula in order to express $E[\widetilde{V}(T-,X^x_u(T-))]$ as

\begin{align*}
E[\widetilde{V}&(T-,X^x_u(T-))]=\\
=&\, \widetilde{V}(0,x) + \int_0^T E\left[\partial_1 \widetilde{V}(t,X^x_u(t)) + \partial_2 \widetilde{V}(t,X^x_u(t)) u(t) +\frac{1}{2}\partial_2^2 \widetilde{V}(t,X^x_u(t))\right]dt\\
=&\, \widetilde{V}(0,x) + \int_0^T E\left[\partial_1 \widetilde{V}(t,X^x_u(t)) + \partial_2 \widetilde{V}(t,X^x_u(t)) v(t,X^x_u(t)) +\frac{1}{2}\partial_2^2 \widetilde{V}(t,X^x_u(t))\right]dt\\
&+ \int_0^T E\left[\partial_2 \widetilde{V}(t,X^x_u(t))\left(u(t)-v(t,X^x_u(t))\right) \right]dt\\
=&\, \widetilde{V}(0,x) + \int_0^T E\left[\partial_2 \widetilde{V}(t,X^x_u(t))\left(u(t)-v(t,X^x_u(t))\right)\right]dt,
\end{align*}
where we used the martingale property of $M$ in the first equality and in the last equality we have used relation \eqref{e: kolmogorov}. Clearly, we have $E[\widetilde{V}(T-,X^x_u(T-))]=E[\widetilde{V}(T,X^x_u(T))]$.

We continue to estimate the last integrand $E\left[\partial_2 \widetilde{V}(t,X^x_u(t))\left(u(t)-v(t,X^x_u(t))\right)\right]$ for $t\in[0,T]$. On $A:=\{v(t,X^x_u(t))\sign(\partial_2 \widetilde{V}(t,X^x_u(t)))=1\}$ we have 
 $$ \partial_2 \widetilde{V}(t,X^x_u(t))\left(u(t)-v(t,X^x_u(t))\right) < 0$$
and hence we get
 \begin{align*}
 E\left[\partial_2 \widetilde{V}(t,X^x_u(t))\left(u(t)-v(t,X^x_u(t))\right)\right] &\leq E\left[\partial_2 \widetilde{V}(t,X^x_u(t))\left(u(t)-v(t,X^x_u(t))\right)1_{A^c}\right] \\
  &\leq 2 E\left[|\partial_2 \widetilde{V}(t,X^x_u(t))|1_{A^c}\right].
\end{align*}  
 
Inserting the density for $X^x_u(t)$ together with Proposition \ref{p: BATOG} yields
 $$ E[\widetilde{V}(T,X^x_u(T))] \leq \tilde V(0,x) + 2 \int_0^T \int_{\R} \beta_{t,1}(0) |\partial_2 \widetilde{V}(t,z)|1_{\{v(t,z)\neq \sign(\partial_2\widetilde V(t,z))\}}dzdt. $$
\end{proof}

\begin{lem}\label{l:Normal control}
  Let $0<h\leq k$, $v:\mathbb R_+\times\mathbb R\rightarrow\{-1,1\}, (t,x) \mapsto -1$ and $\widetilde V$ as in Proposition \ref{p: error estimate}. Then, $\widetilde V$ is infinitely differentiable on $[0,T)\times\mathbb R$ and
   $$ \int_0^T \E[ (\partial_2\widetilde V(s,X^x_u(s))^2 ] ds < \infty $$
  for any $x\in\mathbb R$, $u\in\mathcal A$.
\end{lem}
\begin{proof}
 $\widetilde V$ is obviously infinitely differentiable on $[0,T)\times\mathbb R$.
 
Let $x\in\R$, $u\in\mathcal A$. We have
  \begin{align*}
    \E[ & (\partial_2\widetilde V(s,X^x_u(s))^2 ] \leq \int_{\R} \beta_{s,1}(0) (\partial_2\widetilde V(s,y))^2 dy \\
      &= \beta_{s,1}(0) \int_{\R}\frac{1}{T-s}\Bigg[\phi\left(\frac{h+k-y+T-s}{\sqrt{T-s}}\right)-\phi\left(\frac{k-y+T-s}{\sqrt{T-s}} \right)\\
      &\hspace{4.5cm}-\phi\left(\frac{h-y+T-s}{\sqrt{T-s}} \right)+\phi\left( \frac{-y+T-s}{\sqrt{T-s}}\right)\Bigg]^2dy \\
      &\leq \beta_{s,1}(0) \frac{1}{\sqrt{2(T-s)}}8\phi(0) \\
      &= \frac{\sqrt{8}\beta_{s,1}(0)}{\sqrt{2(T-s)}}\phi( 0 ).
  \end{align*}   
 where we used Proposition \ref{p: BATOG} for the first inequality and Lemma \ref{l:phi square} for the second inequality. The claim follows.
\end{proof}

\begin{prop}\label{p:control problem}
 Let $0< h\leq k$ and $\alpha\in(0,1)$. Then we have
\begin{align*}
 E[J_{h,k}(X^x_u(T))] &\leq hk^\alpha C_\alpha
\end{align*}
 where $$C_\alpha:= \frac{1}{\sqrt{2\pi e^\alpha}T^{(1+\alpha)/2}} + \frac{4}{\sqrt{2\pi e^\alpha}} \int_0^T \left(\frac{\phi(\sqrt s)}{\sqrt{s}}+\Phi(\sqrt{s})\right) \frac{1}{(T-s)^{(1+\alpha)/2}}ds$$ for any $u\in\mathcal A$, $x\in\mathbb R$.
\end{prop}
\begin{proof}
 Let $u\in\mathcal A$, $x\in\mathbb R$ and let $v$ and $\widetilde V$ be as in Lemma \ref{l:Normal control}. Then Lemma \ref{l:Normal control} yields that the requirements of Proposition \ref{p: error estimate} hold. Thus Proposition \ref{p: error estimate} yields
  \begin{align*}
 E[J_{h,k}(X^x_u(T))]\leq \widetilde V(0,x) + 2\int_0^T \int_{\R}\beta_{s,1}(0) 1_{\{1 =   \sgn(\partial_2 \widetilde V(s,z))\}} \partial_2 \widetilde V(s,z) dzds.
\end{align*}
 Lemma \ref{l:normal estimate II} implies
  \begin{align*}
   |\widetilde V(t,x)| &\leq \frac{hk^\alpha}{\sqrt{2\pi e^\alpha}(T-t)^{(1+\alpha)/2}}.
\end{align*} 
 for any $t\in[0,T)$. Moreover, Lemma \ref{l:normal estimate I} yields that there are $a(t),b(t)\in\mathbb R$ with $a(t)\leq b(t)$ such that $\partial_2\widetilde V(t,\cdot)$ is negative on $\R\backslash [a(t),b(t)]$ and positive on $(a(t),b(t))$ for any $t\in[0,T)$. Hence, we get
  \begin{align*}
\int_0^T \int_{\R}&\beta_{s,1}(0) 1_{\{1 = \sgn(\partial_2 \widetilde V(s,z))\}} \partial_2 \widetilde V(s,z) dzds \\
&= \int_0^T \beta_{s,1}(0)\int_{a(s)}^{b(s)} \partial_2 \widetilde V(s,z) dzds \\
 &= \int_0^T \beta_{s,1}(0) (\widetilde V(s,b(s))-\widetilde V(s,a(s))) ds  \\
 &\leq 2\int_0^T \beta_{s,1}(0) \frac{hk^\alpha}{\sqrt{2\pi e^\alpha}(T-s)^{(1+\alpha)/2}} ds \\
 & = 2\frac{hk^\alpha}{\sqrt{2\pi e^\alpha}}\int_0^T \left(\frac{\phi(\sqrt s)}{\sqrt{s}}+\Phi(\sqrt{s})\right) \frac{1}{(T-s)^{(1+\alpha)/2}}ds.
\end{align*}
The claim follows.
\end{proof}

\begin{prop}\label{p: control problem}
 Let $u\in\mathcal A$, $x\in\mathbb R$ and $\alpha\in(0,1)$. Then $X^x_u(T)$ has $\alpha$-H\"older continuous density with constant $C_\alpha$ given in Proposition \ref{p:control problem}.
\end{prop}
\begin{proof}
  Let $F(y) := P(X_u^x(T) \leq y) = P(X_u^{x-y}(T) \leq 0)$ be the distribution function of $X_u^x(T)$. Then $F$ is absolutely continuous and Proposition \ref{p:control problem} yields that
   $$ F(y+h+k)-F(y+k)-F(y+h)-F(y) \leq hk^\alpha C_\alpha $$
 for any $y\in\R$, $0<h\leq k$. Thus Proposition \ref{p:Holder condition} yields that $F$ is continuously differentiable and its derivative is $\alpha$-H\"older continuous with constant $C_\alpha$.
\end{proof}

\appendix
\section{H\"older properties}\label{app 1}
In this section we gather some results on H\"older continuity. We start by recalling the definition.
\begin{defn}\label{d:Holder}
Let $I\subseteq\mathbb R$ be an interval with at least two points and $\alpha\in(0,1]$. A function $f:I \rightarrow \R$ is said to be globally $\alpha$-H\"{o}lder continuous with constant $C$ if
$$\left|f(x)-f(y)\right| \leq C|x-y|^\alpha $$
for any $x,y\in I$.

In this article, we will not use the concept of local H\"older continuous function and, hence, we will always mean globally H\"{o}lder continuous if we say that a function is H\"older continuous.
\end{defn}

Next we give an elementary interpolation result for H\"older continuity.
\begin{lem}\label{l:Holder condition}
 Let $I\subseteq\mathbb R$ be an interval containing at least two points, $\alpha\in (0,1]$, $f:I\rightarrow\mathbb R$ be Lipschitz-continuous and bounded and define
  \begin{align*}
    L_f &:= \sup\left\{ \frac{|f(x)-f(y)|}{|x-y|}: x,y\in I, x\neq y \right\}, \\
    B_f &:= \sup\{ |f(x)-f(y)|: x,y\in I\}.
  \end{align*}
  Then, $f$ is $\alpha$-H\"older continuous and
   $$ |f(x)-f(y)| \leq \left(B_f^{1-\alpha} L_f^{\alpha}\right) |x-y|^\alpha $$ 
   for any $x,y\in I$.
\end{lem}
\begin{proof}
  If $L_f=0$, then $f$ is constant and the claim obviously holds. Thus, we may assume that $L_f>0$ and, hence, $B_f>0$. Define $c:= B_f/L_f$. Let $x,y\in I$.
  
  \underline{Case 1}: $|x-y|\leq c$. Then, we have 
  $$|f(x)-f(y)| \leq L_f|x-y| \leq L_f c \frac{|x-y|^\alpha}{c^\alpha} = B_f^{1-\alpha} L_f^{\alpha} |x-y|^\alpha.$$
  
  \underline{Case 2}: $|x-y|> c$. Then, we have
  $$|f(x)-f(y)| \leq B_f \leq B_f \frac{|x-y|^\alpha}{c^\alpha} = B_f^{1-\alpha} L_f^{\alpha} |x-y|^\alpha. $$
\end{proof}

Corollary \ref{l:Holder condition} applied to the normal density yields a relative simple constant.
\begin{cor}\label{c:Holder cont of phi}
  Let $\alpha\in(0,1]$. Then $\phi$ is $\alpha$-H\"older continuous with constant $C_\alpha:=\frac{1}{\sqrt{2\pi e^\alpha}}$, i.e.\ $|\phi(x)-\phi(y)|\leq \frac{|x-y|^\alpha}{\sqrt{2\pi e^{\alpha}}}$ for any $x,y\in\mathbb R$.  
\end{cor}
\begin{proof}
 This is an immediate consequence of Lemma \ref{l:Holder condition}
\end{proof}

The next proposition gives an exact condition that ensures that the derivative of a function is H\"older continuous. If $X$ is a random variable, then H\"older-continuity of its density can be expressed in terms of weighted differences of certain probabilities.
\begin{prop}\label{p:Holder condition}
 Let $F:\mathbb R\rightarrow\mathbb R$ be absolutely continuous and $\alpha\in(0,1]$. Then the following two statements are equivalent:
 \begin{enumerate}
  \item $F$ is continuously differentiable and $F'$ is $\alpha$-H\"older continuous.
  \item There is $C>0$ such that
  $$ \sup_{x\in\mathbb R}\sup_{k\geq h>0} \left|\frac{F(x+h+k)-F(x+k)-F(x+h)+F(x)}{hk^{\alpha}}\right| \leq C\, . $$
 \end{enumerate}
\end{prop}
\begin{proof}
 $(1)\Rightarrow(2)$: Assume (1) and let $C>0$ be the H\"{o}lder constant of $F'$. We have
 \begin{align*}
   | F(x+h+k)-F(x+k)-F(x+h)+F(x)| &= \left|\int_{x}^{x+h} \left(F'(y+k)-F'(y)\right) dy \right| \\
                                  &\leq Chk^{\alpha}
 \end{align*}
 for any $x\in\mathbb R$, $k>0$, $h>0$ and we used the fundamental theorem of calculus. Thus, (2) follows.
 
 $(2)\Rightarrow(1)$: Assume (2).  Let $f$ be a version of the absolute continuous derivative of $F$, i.e.\ $\int_{x}^{x+h}|f(y)|dy < \infty$ and $\int_x^{x+h}f(y)dy = F(x+h)-F(x)$ for any $x\in\mathbb R$, $h>0$. By Lebesgue's differentiation theorem \cite[Corollary 2.1.17]{grafakos.08} we have
  $$ f(x) = \lim_{h\rightarrow 0}\frac{F(x+h)-F(x)}{h} $$
 for any $x\in A$ where $A\subseteq \mathbb R$ is Borel measurable and the Lesbesgue measure of $\mathbb R\backslash A$ equals zero. In particular, $A$ is dense in $\mathbb R$.

 Let $C>0$ be the given constant, $x\in A$ and $k>0$ such that $x+k\in A$. We show that $|f(x+k)-f(x)|\leq Lk^\alpha$. Let $\epsilon>0$ be arbitrary and choose $h\in(0,k]$ such that
  \begin{align*}
     \left|f(x)-\frac{F(x+h)-F(x)}{h}\right| &\leq \epsilon,\quad\text{and} \\
     \left|f(x+k)-\frac{F(x+h+k)-F(x+k)}{h}\right| &\leq \epsilon.
  \end{align*}
 Thus, we have
 \begin{align*}
   \left| f(x+k)-f(x)\right| &\leq 2\epsilon + \left|\frac{F(x+h+k)-F(x+k)-F(x+h)+F(x)}{h} \right| \\
   &\leq 2\epsilon + Lk^\alpha.
 \end{align*}
 Hence, we have $|f(x)-f(y)|\leq C|x-y|^\alpha$ for any $x,y\in A$. 
 Consequently, there is an $\alpha$-H\"older continuous function $g:\mathbb R\rightarrow\mathbb R$ such that $g(x) = f(x)$ for any $x\in A$. Clearly, $g$ is another version of the absolutely continuous derivative of $F$. The fundamental theorem of calculus yields that $F$ is continuously differentiable and $F'=g$.
\end{proof}

\section{Normal estimates}\label{app 2}
Connected to the last section we find some normal estimates for functions related to Proposition \ref{p:Holder condition}.
\begin{lem}\label{l:normal estimate I}
  Let $Z$ be a normal random variable with mean $\mu\in\mathbb R$ and standard deviation $\sigma>0$ and let $0<h\leq k$. We define
   $$ N(x) := P(Z+x\in[k,k+h]) - P(Z+x\in[0,h]),\quad x\in\mathbb R. $$
 Then $N$ is infinitely differentiable and $N'$ is negative on $(-\infty,-\mu-\sigma)$ and on $(h+k-\mu+\sigma,\infty)$. Moreover, $N'$ is positive if and only if $x\in(a,b)$ where $a,b$ are the two zeros of $N'$ (and $N'$ has exactly two zeros).
\end{lem}
\begin{proof}
  Clearly, $N$ is infinitely differentiable. Thus, we have
   $$ N(x) = \frac{1}{\sigma^2}\int_0^h\int_0^k \phi'\left( \frac{y+z-x-\mu}{\sigma} \right)dydz $$
   for any $x\in\R$. Hence, by dominated convergence we get
   $$ N'(x) = \frac{-1}{\sigma^3}\int_0^h\int_0^k \phi''\left( \frac{y+z-x-\mu}{\sigma} \right)dydz. $$
   Since $\phi''$ is positive on $\R\backslash(-1,1)$ the first claim follows.
   
  If we show that $N'$ has exactly two zeros, then the second claim follows as well. By scaling and shifting we may assume that $\mu=0$ and $\sigma=1$. Then, we have
   \begin{align*}
       N'(x) &= -\phi(x)+\phi(x-h)+\phi(x-k)-\phi(x-h-k)
  \end{align*}
for any $x$. We get
   \begin{align*}
      N'(x) &= ( \phi(x-h)-\phi(x) ) + ( \phi(x-k)-\phi(x-h-k) ) > 0
   \end{align*}
 for any $x\in [h/2,k+h/2]$. Clearly, there is $x_0<0$ and $x_1>h+k$ such that $N'(x_0)<0$ and $N'(x_1)<0$. By the intermediate value theorem there is $a\in [x_0,h/2]$ such that $N'(a)=0$ and $b\in[k+h/2,x_1]$ such that $N'(b)=0$. 
 
 Moreover, we have
  \begin{align*}
    N''(x) = -xN'(x) +h\phi(x-h)+k\phi(x-k)-(h+k)\phi(x-h-k)
  \end{align*}
for any $x\in\R$ and for any zero $\tilde a\in(-\infty,h/2]$ we have
  \begin{align*}
    N''(\tilde a) = h\phi(\tilde a-h)+k\phi(\tilde a-k)-(h+k)\phi(\tilde a-h-k) > 0.
  \end{align*}
   Consequently, $N'$ has exactly one zero in $(-\infty,h/2)$. By symmetry of $N'$ around $(h+k)/2$ we have that $N'$ has exactly one zero in $(k+h/2,\infty)$, namely $b=(h+k)/2-a$.
\end{proof}

\begin{lem}\label{l:normal estimate II}
  Assume the requirements of Lemma \ref{l:normal estimate I} and let $\alpha\in(0,1]$. Then we have
   $$ |N(x)| \leq \frac{hk^\alpha}{\sqrt{2\pi e^{\alpha}}\sigma^{1+\alpha}} $$
\end{lem}
\begin{proof}
 The fundamental theorem of calculus and Corollary \ref{c:Holder cont of phi} yield
  \begin{align*}
    |N(x)| &\leq \frac{1}{\sigma}\int_0^h \left|\phi\left( \frac{y+k-x-\mu}{\sigma} \right)-\phi\left( \frac{y-x-\mu}{\sigma} \right) \right|dy \\
       &\leq \frac{hk^\alpha}{\sqrt{2\pi e^{\alpha}}\sigma^{1+\alpha}}.
\end{align*}   
\end{proof}

\begin{lem}\label{l:phi square}
 Let $\sigma>0$, $y,z\in\mathbb R$. Then we have
  $$ \frac{1}{\sigma^2}\int_{\R}\phi\left(\frac{x-y}\sigma\right)\phi\left(\frac{x-z}\sigma\right)dx = \frac{\phi((y-z)/\sqrt{2\sigma^2})}{\sqrt{2\sigma^2}}. $$
\end{lem}
\begin{proof}
 Using the substitution $u=\frac{x-y}\sigma$ and the notation $a:=\frac{y-z}\sigma$ we get
  \begin{align*}
    \frac{1}{\sigma^2}\int_{\R}\phi\left(\frac{x-y}\sigma\right)\phi\left(\frac{x-z}\sigma\right)dx &= \frac{1}{\sigma} \int_{\R} \phi(u)\phi(u+a)du \\
     &= \frac{1}{\sigma\sqrt{2}}\phi(a/\sqrt{2}) \\
     &= \frac{\phi((y-z)/\sqrt{2\sigma^2})}{\sqrt{2\sigma^2}}.
  \end{align*}
\end{proof}

\end{document}